\documentclass[12pt]{article}
\usepackage{amsmath,amsthm,amsfonts,latexsym,amssymb,mathrsfs,graphics,pst-all}

 \usepackage{footnote}
\theoremstyle{plain}
\newtheorem{theorem}{Theorem}[section]
\newtheorem{lemma}[theorem]{Lemma}
\newtheorem{corollary}[theorem]{Corollary}
\newtheorem{proposition}[theorem]{Proposition}

\theoremstyle{definition}
\newtheorem{definition}[theorem]{Definition}
\newtheorem{example}[theorem]{Example}

\newtheorem{remark}[theorem]{Remark}

\numberwithin{equation}{section}

 \newcommand{\bnum}{\begin{enumerate}}
 \newcommand{\enum}{\end{enumerate}}

\begin{document}

\begin{center}{\bf \large{Symmetricity of rings relative to the prime radical}}\\
\end{center}

\begin{center}
 {\textbf{Debraj Roy}\\ Department of Mathematics, National Institute of Technology Meghalaya, Shillong-793003, India.\\  Email: debraj.hcu@gmail.com\\ 
\textbf{Tikaram Subedi}\\ Department of Mathematics, National Institute of Technology Meghalaya, Shillong-793003, India. \\  Email: tsubedi2010@gmail.com}
\end{center}

\begin{small}
\textbf{Abstract}:\textit{In this paper, we introduce and study a strict generalization of symmetric rings. We call a  ring $R$ \textit{`$P$-symmetric' } if for any $a,\, b,\, c\in R,\, abc=0$ implies $bac\in P(R)$, where $P(R)$ is the prime radical of $R$. It is shown that the class of $P$-symmetric rings lies between the class of central symmetric rings and generalized weakly symmetric rings. Relations are provided between $P$-symmetric rings and some other known classes of rings. From an arbitrary $P$-symmetric ring, we produce many families of $P$-symmetric rings. }


\end{small}

\section{Introduction}\label{Sec1}

Throughout this paper, $R$ denotes an associative ring with identity and all modules are unitary. 
The symbols  $E$($R$), $J(R)$, $N$($R$), $P(R)$, $Z(R)$ respectively stand for the set of all idempotent elements, the Jacobson radical, the set of all nilpotent elements, the prime radical and the center of $R$. $R$ is \textit{reduced} if $N(R)=0$. $R$ is \textit{left (right) quasi-duo} if every maximal left (right) ideal of $R$ is an ideal. An ideal $P$ of $R$ is \textit{prime} if for any ideals $A,\, B$ of $R$ with $AB\subseteq P$, either $A\subseteq P$ or $B\subseteq P$. An element $a\in R$ is \textit{strongly nilpotent} if we consider any sequence $\lbrace p_{n}\rbrace $ where $ p_0=a$ and $p_{i+1}\in p_{i}Rp_{i}$ for all $i\geq 0$, then there exists a positive integer $k$ such that $p_{k}=0$. It is well known that 
$P(R)$ consists of all strongly nilpotent elements of $R$. We also know that $P(R)=\lbrace a\in R\mid RaR$ is nilpotent$\rbrace$. $R$ is \textit{2-primal} if $N(R)=P(R)$.

 $R$ is \textit{symmetric} if for any $a,\, b,\, c\in R,\, abc=0$ implies $bac=0$. Lambek in \cite{lam} introduced symmetric rings and obtained some of the significant results in this direction.  Further contribution to symmetric rings and their generalizations have been made by various authors over the last several years (see, \cite{kafkas}, \cite{lam}, \cite{ouyang}, \cite{wei1}). Recently, semicommutativity of rings related to the prime radical was studied in \cite{kose1}. This motivated us to introduce rings called  $P$-symmetric  rings wherein a ring $R$ is called $P$-symmetric if for any  $a,\, b,\, c\in R,\, abc=0$ implies $bac\in P(R)$. This paper studies $P$-symmetric rings in consultation and continuation with various existing generalizations of symmetric rings.

\section{$P$-symmetric rings}

\begin{definition}
We call a ring $R$ \textit{`$P$-symmetric'} if for any $a,\, b,\, c\in R,\, abc=0$ implies $bac\in P(R)$.
\end{definition}

It follows that symmetric rings are $P$-symmetric. Not every $P$-symmetric ring is symmetric  as shown by the following example:

\begin{example}\label{expr}
Let $R=S_4(\mathbb{R})=\left \lbrace
\left(\begin{array}{lccccr}
a & b_{1} & b_{2} & b_{3}\\
0 & a & b_{4} & b_{5}\\
 0 & 0 & a & b_{6}\\
 0 & 0 & 0 &  a
\end{array}
\right ):a,\, b_{i}\in {\mathbb{R}} \right \rbrace$. It is easy to see that every element of $R$ is either a unit or an element of the prime radical and so $R$ is $P$-symmetric.

Now $A=\left(\begin{array}{lccccr}
	0 & 0 & 0 & 0\\
	0 & 0 & 1 & 0\\
           0 & 0 & 0 & 0\\
           0 & 0 & 0 & 0
	\end{array}
	\right )$, $B=\left(\begin{array}{lccccr}
	0 & 1 & 0 & 0\\
	0 & 0 & 0 & 0\\
           0 & 0 & 0 & 0\\
           0 & 0 & 0 & 0
	\end{array}
	\right )\in R$ and $AB=0$ but $BA\neq 0$ which proves that $R$ is not symmetric.

\end{example}

$R$ is \textit{$P$-semicommutative} (\cite{kose1}) if for any $a,\, b\in R,\, ab=0$ implies $aRb\subseteq P(R)$.

\begin{theorem}\label{prchar}
 Let $R$ be a $P$-symmetric ring. Then $N_2(R)=\lbrace a\in R \mid a^2=0\rbrace \subseteq P(R)$. In particular, $R$ is $P$-semicommutative.
\end{theorem}
\begin{proof}
  Let $R$ be a $P$-symmetric ring and $a\in N_{2}(R),\, r\in R$. Then $raa=0$. As $R$ is $P$-symmetric, we obtain $ara\in P(R)$. Therefore $aRa\subseteq P(R)$ which leads to $a\in P(R)$. By (\cite{kose1}, Theorem 2.4), $R$ is $P$-semicommutative.
\end{proof}

\begin{theorem}\label{2pr}
The following conditions are equivalent for any ring $R$:
\begin{enumerate}
\item $R$ is $2$-primal.
\item For any $a,\, b\in R,\, ab\in P(R)$ implies $ba\in P(R)$.
\item $R/P(R)$ is reduced.
\end{enumerate}
\end{theorem}
\begin{proof}
$(1)\Longrightarrow (2)$.
Let $a,\, b\in R$ with $ab\in P(R)$. Then $(ba)^2=b(ab)a\in P(R)=N(R)$ which implies that $ba\in N(R)=P(R)$.

$(2)\Longrightarrow (3)$.
Let $a\in R$ with $a^2\in P(R)$. Then  for any $r\in R,\,  raa\in P(R)$ and hence by hypothesis, $ara\in P(R)$. Therefore $a\in P(R)$.

$(3)\Longrightarrow (1)$ is trivial.
\end{proof}

\begin{theorem}
The following conditions are equivalent for a $2$-primal ring $R$:
\begin{enumerate}
\item $R$ is $P$-symmetric.
\item For any $a,\, b,\, c\in R,\, abc=0$ implies $acb\in P(R)$.
\item For any $a,\, b,\, c\in R,\, abc=0$ implies $cba\in P(R)$.
\end{enumerate}
\end{theorem}
\begin{proof}
$(1)\Longrightarrow (2)$.
Let $a,\, b,\, c\in R$ with $abc=0$. By hypothesis, $(acb)^2=ac(bac)b\in P(R)$. Then by Theorem \ref{2pr}, $acb\in P(R)$.

That $(2)\Longrightarrow (3)$ and $(3)\Longrightarrow (1)$ can be proved similarly.
\end{proof}

\begin{theorem}
Let $R$ be a left quasi-duo ring such that every prime ideal of $R$ is maximal. Then $R$ is $P$-symmetric.
\end{theorem}
\begin{proof}
We note that $J(R)=P(R)$ since every prime ideal of $R$ is maximal. Let $a,\, b,\, c\in R$ with $abc=0$ and $M$ be a maximal left ideal of $R$. If $a\notin M$, then $x+ya=1$ for some $x\in M,\, y\in R$ leading to $xbc=bc$. Since $R$ is left quasi-duo, this leads to $bc\in M$.
 If $b\notin M$, then $(1-qb)c\in M$ for some $q\in R$ which further leads to $c\in M$. It follows that $bac\in J(R)=P(R)$.
\end{proof}

$R$ is \textit{central symmetric} (\cite{kafkas}) if for any $a,\, b,\, c\in R,\, abc=0$ implies $bac\in Z(R)$. $R$ is \textit{generalized weakly symmetric} (\cite{wei1}) if for any $a,\, b,\, c\in R,\, abc=0$ implies $bac\in N(R)$. 

\begin{theorem}
Every central symmetric ring is $P$-symmetric.
\end{theorem}
\begin{proof}
Let $R$ be central symmetric and $a,\, b,\, c\in R$ with $abc=0$. As every central symmetric ring is generalized weakly symmetric (\cite{wei1}, Proposition 2.3), there exists a positive integer $m$ such that $(bac)^{2^{m}}=0$. Consider any sequence $\lbrace p_{n}\rbrace $ where $ p_0=bac$ and $p_{i+1}\in p_{i}Rp_{i}$ for all $i\geq 0$. Since $bac\in Z(R),\, p_1=(bac)^{2}r_1$ for some $r_1\in R$. Similarly $p_2=(bac)^{4}r_2$ for some $r_2\in R$. Therefore it can be shown that for any positive integer $n,\, p_{n}=(bac)^{2^{n}}r_{n}$ for some $r_{n}\in R$. Hence it follows that $p_{m}=0$. Therefore $bac\in P(R)$.
\end{proof}

Not every $P$-symmetric ring is central symmetric as shown by the following example:

\begin{example}\label{ex}
Let $R=S_4(\mathbb{R})$. Then $R$ is $P$-symmetric.

Take $A=\left(\begin{array}{lccccr}
    0 & 0  & 0 & 0\\
	0 & 0 & 1 & 0\\
	0 & 0 & 0 & 0\\
	0 & 0 & 0 & 0
	\end{array}
	\right )$ and $B=\left(\begin{array}{lccccr}
	0 & 1  & 0 & 0\\
	0 & 0 & 0 & 0\\
	0 & 0 & 0 & 0\\
	0 & 0 & 0 & 0
	\end{array}
	\right )\in R$. Then $AB=0$ but $BA\notin Z(R)$ so that $R$ is not central symmetric.
	\end{example}

Observing that  $P(R)\subseteq N(R)$, we have the following theorem:

\begin{theorem}
Every $P$-symmetric ring is generalized weakly symmetric.
\end{theorem}

$R$ is $\textit{weakly reversible}$ {\rm (\cite{gang})} if for any $a, b, r\in R, ab=0$ implies $Rbra$ is a nil left ideal.

\begin{proposition}
Every $P$-symmetric ring is weakly reversible.
\end{proposition}
\begin{proof}
 Let $R$ be a $P$-symmetric ring and $a, b, r\in R$ with $ab=0$. For any $s \in R,\, (sbra)(bra)(sbra)=0$. By hypothesis, $bra(sbra)^2\in P(R)\subseteq N(R)$ which implies that $sbra\in N(R)$. Hence $Rbra$ is a nil left ideal.
\end{proof}

\begin{remark}
Since a homomorphic image of a central symmetric ring need not be generalized weakly symmetric (\cite{wei1}, Example 2.11), it follows that a homomorphic image of a $P$-symmetric ring need not be P-symmetric.	
\end{remark}

\begin{proposition}
Let $R$ be a ring and $e\in E(R)$. If $R$ is $P$-symmetric, then $eRe$ is $P$-symmetric.
\end{proposition}
\begin{proof}
The result follows from the fact that for any ring $R,\, P(eRe)=eP(R)e$ for any $e\in E(R)$ (\cite{coy}).
\end{proof}

\begin{corollary}
For any ring $R$, $R/P(R)$ is $P$-symmetric implies $R$ is $P$-symmetric.
\end{corollary}

\begin{lemma}\label{prsub}
(\cite{kose1}, Lemma 3.2) Let $R$ be a ring and $I,\, J$ are ideals of $R$ with $I\cap J=0$. Then $P(R)=(\bigcap_{i\in I_{1}}P_{i})\bigcap (\bigcap_{i\in I_{2}}P_{i})$ 
and $P_{i}$ is a prime ideal of $R$ for every $i\in I_{1}\bigcup I_{2}$ where $I_{1}$ and $I_{2}$ are index sets for the prime ideals of $R$ containing $I$ and $J$, respectively.
\end{lemma}

\begin{theorem}\label{subpr}
Finite subdirect product of  $P$-symmetric rings is $P$-symmetric.
\end{theorem}
\begin{proof}
 Let $R$ be the subdirect product of two $P$-symmetric rings $A$ and $B$. Then we have epimorphisms $f: R\rightarrow A$ and $g: R\rightarrow B$ with $Ker(f)\cap  Ker(g)=0$ and $A\cong R/Ker(f)$ and $B\cong R/Ker(g)$.
 We denote $I=Ker(f),\, J=Ker(g)$. Let $a,\, b,\, c\in R$ with $abc=0$. Then $\overline{a}\overline{b}\overline{c}=\overline{0}\in R/I$. 
Since $R/I\cong A$ is $P$-symmetric, $\overline{b}\overline{a}\overline{c}\in P(R/I)=(\bigcap_{i\in I_{1}}P_{i})/I$ where $I_{1}$ is the index set for the prime ideals of $R$ containing $I$. Therefore $bac\in \bigcap_{i\in I_{1}}P_{i}$. 
Similarly we can prove that $bac\in \bigcap_{i\in I_{2}}P_{i}$ where $I_{2}$ is the index set for the prime ideals of $R$ containing $J$. Hence by Lemma \ref{prsub}, $bac\in P(R)$ which proves that $R$ is $P$-symmetric.
\end{proof}

\begin{lemma}\label{l1}
(\cite{kose1}, Lemma 2.17) Let $R$ be a ring and $S$ be a multiplicatively closed subset of $R$ consisting of central regular elements. Then $P(S^{-1}R)=\lbrace u^{-1}a\mid u\in S,\, a\in P(R) \rbrace$.
\end{lemma}

\begin{theorem}\label{rof}
Let $R$ be a ring and $S$ be a multiplicatively closed subset of $R$ consisting of central regular elements. Then $R$ is $P$-symmetric if and only if $S^{-1}R$ is $P$-symmetric.
\end{theorem}
\begin{proof}
Let $R$ be a $P$-symmetric ring and $\alpha,\, \beta,\, \gamma\in S^{-1}R$ with $\alpha\beta\gamma=0$. Let $\alpha=m^{-1}a,\, \beta=n^{-1}b,\, \gamma=p^{-1}c$ where $m,\, n,\, p\in S,\, a,\, b,\, c\in R$. 
Since $S\subseteq Z(R),\, \alpha\beta\gamma=m^{-1}an^{-1}bp^{-1}c=(mnp)^{-1}abc=0$, so that $abc=0$. As $R$ is $P$-symmetric, $bac\in P(R)$. Therefore by Lemma \ref{l1}, $\beta\alpha\gamma \in P(S^{-1}R)$ which implies that $S^{-1}R$ is $P$-symmetric.\\
Converse is trivial.
\end{proof}

 $R$ is \textit{Armendariz} (\cite{sima}) if for any $f(x)=\displaystyle \sum_{i=0}^{i=m} a_{i}x^{i},\, g(x)=\displaystyle \sum_{j=0}^{j=n} b_{j}x^{j}\in R[x],\, f(x)g(x)=0$ implies  $a_{i}b_{j}=0$ for every $i,\, j$.\\

\begin{theorem}
Consider the following statements for any ring $R$:
\begin{enumerate}
\item  $R$ is $P$-symmetric.
\item $R[x]$ is $P$-symmetric.
\item The ring of Laurent polynomials $R[x;x^{-1}]$ is $P$-symmetric.
\end{enumerate}
Then $(2)\Longrightarrow (3) \Longrightarrow (1)$. Further, $(1)\Longrightarrow (2)$ if $R$ is an Armendariz ring.
\end{theorem}

\begin{proof}
$(2)\Longrightarrow (3)$.
Assume $R[x]$ is $P$-symmetric and let $S=\lbrace 1,x,x^{2},...\rbrace$. Then $S$ is a multiplicatively closed subset of $R[x]$ consisting of central regular elements. Therefore by Theorem \ref{rof}, $S^{-1}R[x]$ is $P$-symmetric. 
Since $R[x;x^{-1}]\simeq S^{-1}R[x]$, the result follows.\\
$(3)\Longrightarrow (1)$ is trivial.

Let $R$ be an Armendariz ring and $f(x)=\displaystyle \sum_{i=0}^{i=m} a_{i}x^{i},\, g(x)=\displaystyle \sum_{j=0}^{j=n} b_{j}x^{j},\, h(x)=\displaystyle \sum_{k=0}^{k=l} c_{k}x^{k}\in R[x]$ with $f(x)g(x)h(x)=0$. 
Since $R$ is Armendariz, by (\cite{dd}, Proposition 1), $a_{i}b_{j}c_{k}=0$ for all $i,\, j,\, k$. As $R$ is $P$-symmetric, $b_{j}a_{i}c_{k}\in P(R)$ for all $i,\, j,\, k$, which implies that $g(x)f(x)h(x)\in P(R[x])$ as $P(R[x])=P(R)[x]$.

\end{proof}

\begin{theorem}\label{t2}
The following conditions are equivalent for any ring $R$:
\begin{enumerate}
\item  $R$  is $P$-symmetric.

\item $T_n(R)$, the ring of all $n\times n$ upper triangular matrices over $R$ is $P$-symmetric for any $n\geq 1$.

\item  $S_n(R)=\left \lbrace
\left(\begin{array}{lccccr}
a & a_{12} & \dots & a_{1n}\\
0 & a & \dots & a_{2n}\\
\vdots & \vdots &\ddots & \vdots\\
0 & 0 & \dots & a \\
\end{array}
\right ): a, a_{ij}\in R,  i<j\leq n \right \rbrace$ is $P$-symmetric for any $n\geq 1$.

\item $V_n(R)=\left \lbrace
\left(\begin{array}{lccccr}
a_{0} & a_{1} & a_{2} \dots &  a_{n-1}\\
0 & a_{0} &  a_{1} \dots & a_{n-2}  \\
\vdots & \vdots & \vdots\ddots & \vdots \\
0 & 0 & 0 \dots &a_{0} \\
\end{array}
\right ): a_i\in R,  i=0,1,2,..., n-1 \right \rbrace$  is $P$-symmetric for any $n\geq 1$.
\end{enumerate}
\end{theorem}

\begin{proof}
That $(2)\Longrightarrow (1)$,\, $(3)\Longrightarrow (1)$,\, $(4)\Longrightarrow (1)$ follows trivially.

We know that for any $n\geq 1$,\, 

$P(T_n(R))=\left \lbrace
\left(\begin{array}{lccccr}
a_{11} & a_{12} & \dots & a_{1n}\\
0 & a_{22} & \dots & a_{2n}\\
\vdots & \vdots &\ddots & \vdots\\
0 & 0 & \dots & a_{nn} \\
\end{array}
\right ): a_{ii}\in P(R),\,a_{ij} (i\neq j)\in R \right \rbrace$,

 $P(S_n(R))=\left \lbrace
\left(\begin{array}{lccccr}
a & a_{12} & \dots & a_{1n}\\
0 & a & \dots & a_{2n}\\
\vdots & \vdots &\ddots & \vdots\\
0 & 0 & \dots & a \\
\end{array}
\right ): a\in P(R), a_{ij}\in R,  i<j\leq n \right \rbrace$,

$P(V_n(R))=\left \lbrace
\left(\begin{array}{lccccr}
a_{0} & a_{1} & a_{2} \dots &  a_{n-1}\\
0 & a_{0} &  a_{1} \dots & a_{n-2}  \\
\vdots & \vdots & \vdots\ddots & \vdots \\
0 & 0 & 0 \dots &a_{0} \\
\end{array}
\right ): a_{0}\in P(R), a_i\in R,  i=1,2,..., n-1\right \rbrace$.

$(1)\Longrightarrow (2)$.

Let $A=\left(\begin{array}{lccccr}
a_{11} & a_{12} & \dots & a_{1n}\\
0 & a_{22} & \dots & a_{2n}\\
\vdots & \vdots &\ddots & \vdots\\
0 & 0 & \dots & a_{nn} \\
\end{array}
\right ),\, B=\left(\begin{array}{lccccr}
b_{11} & b_{12} & \dots & b_{1n}\\
0 & b_{22} & \dots & b_{2n}\\
\vdots & \vdots &\ddots & \vdots\\
0 & 0 & \dots & b_{nn} \\
\end{array}
\right ),\\
C=\left(\begin{array}{lccccr}
c_{11} & c_{12} & \dots & c_{1n}\\
0 & c_{22} & \dots & c_{2n}\\
\vdots & \vdots &\ddots & \vdots\\
0 & 0 & \dots & c_{nn} \\
\end{array}
\right )\in T_{n}(R)$ with $ABC=0$. Then for all $i,\, 1\leq i\leq n,\,a_{ii}b_{ii}c_{ii}=0$ and hence by hypothesis, $b_{ii}a_{ii}c_{ii}\in P(R)$  which implies that $BAC\in P(T_n(R))$.

That $(1)\Longrightarrow (3),\, (1)\Longrightarrow (4)$ can be proved in a similar way.
\end{proof}

If $R$ is $P$-symmetric, then $M_n(R)$, the ring of $n\times n$ matrices over $R$, need not be $P$-symmetric as shown by the following example:

\begin{example}
Let $R=M_2(\mathbb{R})$ and $A=\left(\begin{array}{lccccr}
	0 & 0 \\
	1 & 0 \\
	\end{array}
	\right )$, $B=\left(\begin{array}{lccccr}
	0 & 1 \\
	0 & 0 \\
	\end{array}
	\right )$, $C=\left(\begin{array}{lccccr}
	1 & 1\\
	0 & 0 \\
	\end{array}
	\right )\in R$. Then $ABC=0$ but $BAC=\left(\begin{array}{lccccr}
	1 & 1\\
	0 & 0 \\
	\end{array}
	\right )\notin P(R)$ as $BAC$ is not nilpotent.
\end{example}

For any non-empty sets $A$ and $B$, let $R[A, B]$ denote the set $\lbrace (a_1,\, a_2,..., a_{n}, b, b,...): a_{i}\in A,\, b\in B, n\geq 1,\, 1\leq i\leq n\rbrace$. If $A$ is a ring with identity and $B$ is a subring of $A$ with the same identity element of $A$, then $R[A, B]$ becomes a ring.

\begin{lemma}\label{prrab}
(\cite{kose1}, Lemma 3.7) Let $B$ be a subring of a ring $A$.  Then $P(R[A, B])=R[P(A), P(A)\bigcap P(B)]$.
\end{lemma}

\begin{theorem}
Let $B$ be a subring of a ring $A$ with the identity element same as that of $A$. The following statements are equivalent:
\begin{enumerate}
\item $A$ and $B$ are $P$-symmetric.
\item $R[A, B]$ is $P$-symmetric.
\end{enumerate}
\end{theorem}
\begin{proof}
$(1)\Longrightarrow (2)$.
Let  $f,\, g,\, h\in R[A, B]$ satisfy $fgh=0$.\\
Let $f=(a_1,\, a_2,..., a_{n_1},\, a, a,...),\, g=(b_1,\, b_2,..., b_{n_2},\, b, b,...),\\
h=(c_1,\, c_2,..., c_{n_3},\, c, c,...)$, where $a_{i},\, b_{j},\, c_{k}\in A,\, a, b, c\in B,\, n_1,n_2, n_3\geq 1,\, 1\leq i\leq n_1,\, 1\leq j\leq n_2, 1\leq k\leq n_3.$ Take $n=max\lbrace n_1, n_2, n_3\rbrace$. 
If $n_1$ is maximum, let $b_j=b$ for $n_2+1\leq j\leq n_1$, and $c_k=c$ for $ n_3+1\leq k\leq n_1$. Similar relations are assumed  when $n_2$ or $n_3$ is maximum. Then $abc=0$ and for $1\leq i\leq n$, $a_{i}b_{i}c_{i}=0$. 
Therefore by hypothesis and Proposition \ref{prrab}, we conclude that $gfh\in P(R[A, B])$.

$(2)\Longrightarrow (1)$.
Let $a,\, b,\, c\in A$ satisfy $abc=0$. Consider the element $f=(a,0,0,...),\, g=(b,0,0,...),\, h=(c,0,0,...)\in R[A, B]$ with $fgh=0$ in $R[A, B]$. Then by hypothesis and Proposition \ref{prrab}, $gfh\in P(R[A, B])$ which yields $bac\in P(A)$. Hence $A$ is $P$-symmetric. Similarly,  we can establish that $B$ is $P$-symmetric.

\end{proof}

\begin{theorem}
The following statements are equivalent for a ring $R$:
\begin{enumerate}
\item $R$ is $P$-symmetric.
\item The ring $S=\lbrace (x,y)\in R\times R\mid x-y\in P(R) \rbrace$ is $P$-symmetric.
\end{enumerate}
\end{theorem}
\begin{proof}
$(1)\Longrightarrow (2)$.
Consider the homomorphisms $f: S\rightarrow R$ by $(x, y)\rightarrow x$ and $g: S\rightarrow R$ by $(x, y)\rightarrow y$. Then $f$ and $g$ are epimorphisms and $Ker(f)\cap  Ker(g)=0$. 
By hypothesis, $S/Ker(f)\cong R$ and $S/Ker(g)\cong R$ are $P$-symmetric rings. Therefore $S$ becomes a subdirect product of $S/Ker(f)$ and $S/Ker(g)$. Hence by Theorem \ref{subpr}, $S$ is $P$-symmetric.

$(2)\Longrightarrow (1)$.
Let $a,\, b,\, c\in R$ with $abc=0$. Then $(a, a)(b, b)(c, c)=(0, 0)$. By hypothesis, $(b, b)(a, a)(c, c)\in P(S)$. Consider any sequence  $\lbrace p_{n}\rbrace $  in $R$ with $ p_0=bac$ and $p_{i+1}\in p_{i}Rp_{i}$ for all $i\geq 0$. 
Let $ q_{0}=(b, b)(a, a)(c, c),\, q_{1}= (p_1, p_1),\, q_{2}= (p_2, p_2),..., q_{n}= (p_n, p_n)$ with $(p_{i+1}, p_{i+1})= (p_{i}, p_{i})(x,x)(p_{i}, p_{i})$ for all $i\geq 0$, for some $x\in R$. By hypothesis, there exists positive integer $m$ such that $q_{m}=(0, 0)$  which implies that $p_{m}=0$. This shows that $bac\in P(R)$. Hence $R$ is $P$-symmetric.
\end{proof}

\end{document}